\documentclass[10pt,twoside]{article}

\usepackage{amsmath}
\usepackage{amsfonts}
\usepackage{amssymb}
\usepackage{amscd}
\usepackage{amsthm}
\usepackage{amsbsy}
\usepackage{graphicx}
\usepackage{bm}

\usepackage{esint}
\usepackage{enumerate}
\usepackage{tikz}
\usetikzlibrary{patterns,shapes,calc,through,arrows,fadings,decorations.pathreplacing,intersections}

\def\@oddhead{\hfill \shorttitle \hfill \thepage}
\def\@evenhead{\thepage \hfill \shortauthor \hfill}
\def\@oddfoot{}
\def\@evenfoot{}

\def\eps{\varepsilon}
\def\real{\mathbb{R}}

\def\essinf{\mathop{\mbox{\normalfont ess inf}}\limits}
\def\esssup{\mathop{\mbox{\normalfont ess sup}}\limits}
\def\rad{\mbox{\scriptsize{\normalfont{rad}}}}
\def\loc{\mbox{\scriptsize{\normalfont{loc}}}}

\newtheorem{teo}{Theorem}[section]

\newtheorem{lemma}[teo]{Lemma}
\newtheorem{prop}[teo]{Proposition}

\theoremstyle{definition}
\newtheorem*{xrem}{Remark}

\numberwithin{equation}{section}

\date{}
\title{\ \\[0.4cm] \ \\ \bf  On the dimension dependence of some weighted inequalities }
\author{Alberto Criado\footnote{Departamento de Matem\'aticas, Universidad Aut\'onoma de Madrid, 28049 Madrid, Spain.}\hspace{2mm}
and Fernando Soria\footnote{Departamento de Matem\'aticas and Instituto de Ciencias Matem\'aticas CSIC-UAM-UC3M-UCM, Universidad Aut\'onoma de Madrid, 28049 Madrid, Spain.}}
\begin{document}
%-------------------

\maketitle

%------------------------

\thispagestyle{empty}

%--------------------------------------

\hfill{\small{Dedicated with admiration to Professor Shanzhen Lu.}

%--------------------------------------
\begin{abstract}
\vskip 3mm\footnotesize{

\vskip 4.5mm
\noindent
In the context of radial weights we study the dimension dependence of some weighted inequalities for maximal operators. We study the growth of the $A_1$-constants for radial weights and show the equivalence between the uniform boundedness of these constants, a dimension-free weak $L^1$ estimate for the maximal operator on annuli and the condition on the weight to be decreasing and essentially constant over dyadic annuli. Each one of these conditions is shown to provide dimension-free weighted weak type $L^1$ estimates for the centred maximal Hardy-Littlewood operator acting on radial functions. Finally we show that the universal maximal operator is of restricted weak type on weighted $L^n(\real^n)$ with constants uniformly bounded in dimension whenever we consider an $A_1$ weight.

\vspace*{2mm}
\noindent{\bf 2000 Mathematics Subject Classification: 42B25}

\vspace*{2mm}
\noindent{\bf Keywords and Phrases: radial weights, dimension-free estimates}}

\end{abstract}

% ----------------------------------------------------------------

\section{Introduction.}

\bigskip

In this paper we will study the dimension dependence of the bounds for some maximal operators when acting over weighted spaces. First, we consider the centered Hardy-Littlewood maximal operator over Euclidean balls. For a locally integrable $g$ on $\real^n$ it is defined as
\[
Mg(x)=\sup_{R>0} \fint_{B_R(x)} |g(y)|\,dy,
\]
where $B_R(x)$ is the Euclidean ball of radius $R$ centered at $x$. It is well-known since the time of Hardy and Littlewood that for each $n$, this operator is bounded on $L^p(\real^n)$ for $p>1$ and weakly bounded on $L^1(\real^n)$. Much later, E.M. Stein raised the questions whether the operator norm in these inequalities could be bounded independently of the dimension and whether this uniformity in dimension could be related to an infinite-dimensional phenomenon. As far as we are concerned, very little is known about the second question. As for the first one, Stein himself showed in \cite{Stein1} (details in \cite{SteinStromberg}) that for all $p>1$ one has $\|M\|_{L^p(\real^n)\rightarrow L^p(\real^n)}\leq C_p$, with $C_p$ independent of $n$. In joint work with J.O. Str\"omberg \cite{SteinStromberg} he also proved that $\|M\|_{L^1(\real^n)\rightarrow L^{1,\infty}(\real^n)}=O(n)$ as $n\rightarrow \infty$. Although this does not solve the still open problem of deciding whether these weak $L^1$ operator norms grow to infinity with the dimension or not, it is still the best known result.

\bigskip

This problem of uniform bounds in dimension has also been studied for maximal functions where the averages are taken over balls given by arbitrary norms in $\real^n$. In all cases one has $\|M\|_{L^p(\real^n)\rightarrow L^p(\real^n)}\leq C$ with $C$ depending only on $p$ for $p>3/2$ (see \cite{Bourgain1}, \cite{Bourgain2}, \cite{Bourgain3} and \cite{Carbery}). For the balls given by the $\ell^q$ metrics in $\real^n$, with $1\leq q\leq \infty$, for all $p>1$ one has $\|M\|_{L^p(\real^n)\rightarrow L^p(\real^n)}\leq C_{p,q}$ with $C_{p,q}$ independent of $n$ (see \cite{Muller} for the case $1\leq q<\infty$ and \cite{Bourgain5} for $q=\infty$). As for the weak $L^1$ inequalities in this case, Stein and Str\"omberg proved that, in general, $\|M\|_{L^1(\real^n)\rightarrow L^{1,\infty}(\real^n)}=O(n\log n)$ as $n\rightarrow \infty$. It is still unknown if these operator norms remain bounded in all dimensions, except when those averages are taken over the balls of the $\ell^\infty$ metric, that is, the case of cubes with sides parallel to the coordinate axes. In this case Aldaz showed in \cite{Aldaz2} that these weak $L^1$ operator norms grow to infinity with the dimension (see also \cite{Aubrun}).

\bigskip

A variation of the problem arises when the maximal operator is defined using measures different from the Lebesgue one. For a Radon measure $\mu$ on $\real^n$ we define the associated maximal operator
\[
M_\mu g(x)=\sup_{\begin{array}{c}\\[-6mm]\scriptstyle{R>0}\\[-1.5mm] \scriptstyle{\mu(B_R(x))>0}\end{array}} \frac1{\mu(B_R(x))} \int_{B_R(x)} |g(y)|\,d\mu(y).
\]
If $\mu$ has a radial density $w$, then $\mu$ and $M_\mu$ can be defined in all dimensions. We may ask then if the operator norm of $M_\mu$ in $L^p(\mu)$ is uniformly bounded in dimension. If $\mu$ is finite this is not true in general (see \cite{Aldaz}, \cite{Criado}, \cite{AldazPerezLazaro}). For instance, when $\mu$ is the Gaussian measure, the $L^p(\mu)$ operator norms of $M_\mu$ grow exponentially to infinity with the dimension for all $p<\infty$ (see \cite{CriadoSjogren}). The situation is very different when $\mu$ satisfies a doubling condition. One says that $\mu$ is uniformly strong $n$-microdoubling if there exist $K>0$ and $N>0$ so that for all $n\geq N$, $x\in\real^n$, $R>0$ and $y\in B_R(x)$ one has 
\[
\mu(B_{(1+1/n)R}(x))\leq K \mu(B_R(x))\ \mbox{ and }\ \mu(B_R(y))\leq K \mu(B_R(x)).
\] 
Roughly speaking this property guarantees that small dilations and translation do not alter essentially the measure of a ball. If $\mu$ is such a measure, then one recovers the Stein and Str\"omberg bound $\|M_\mu\|_{L^1(\real^n,d\mu)\rightarrow L^{1,\infty}(\real^n,d\mu)}=O(n\log n)$ (see \cite{NaorTao}). Moreover one has the uniform bound $\|M_\mu\|_{L^p(\real^n,d\mu)\rightarrow L^p(\real^n,d\mu)}\leq C_{p, \mu}$ for all $n\geq N$ and $p>1$ (see \cite{CriadoSoria}).

\bigskip

Still another problem is the one that considers weighted inequalities for the maximal operator. A weight $w$ is an a.e. nonnegative and locally integrable function over $\real^n$. A weight is often regarded as the density of a measure over $\real^n$ that is absolutely continuous with respect to the Lebesque measure. Following the usual notation we will also denote this measure by $w$, i.e. for a measurable $E$ we will write $w(E)=\int_Ew$ and we will say that a function $f\in L^p(\real^n,w)$ if $\int_{\real^n}|f|^pw<\infty$. For $p\geq 1$ we say that a weight $w$ is in the class $A_p(\real^n)$ if $M$ sends $L^p(\real^n,w)$ into $L^{p,\infty}(\real^n,w)$ boundedly\footnote{In what follows, the expression $T:X\rightarrow Y$ will denote that the (sublinear) operator T is bounded between the spaces $X$ and $Y$.}. It is well-known that for $p>1$ this is equivalent to $M:L^p(\real^n,w)\rightarrow L^p(\real^n,w)$ boundedly. These bounds have been studied extensively. For more information see \cite{CuervaRubiodeFrancia} or \cite{Grafakos}.

\bigskip

In this work we will consider radial weights, so that we can define them in all dimensions. For such a weight $w$ we will write $w(x)=w_0(|x|)$ with $w_0:[0,\infty)\rightarrow [0,\infty]$. J. Duoandikoetxea and L. Vega announced in \cite{DuoandiVega} the following result.

\bigskip

\begin{teo}\label{teo.duoandi.vega} Let $w_0$ be a nonnegative function on $[0,\infty)$, so that $w=w_0(|\cdot|)\in A_p(\real^N)$ with $p<1$. Then for all $n\geq N$ one has $w\in A_p(\real^n)$ and, moreover, 
\[
\|Mf\|_{L^p(\real^n,w)} \leq C \|f\|_{L^p(\real^n,w)},
\]
with a constant $C$ that might depend on $p$ and $w_0$ but not on $n$.
\end{teo}

\bigskip

This solves completely the problem of finding weighted $L^p$ bounds that are uniform in dimension for $p>1$. In this paper we present some partial results in the case $p=1$. In the next section we present some result relating the growth of the $A_1$ constant of a weight with the uniformity of the weak $L^1$ weighted bounds for the maximal operator over radial functions. The proofs of these results are contained in Section \ref{seccion.pruebas}. Finally, Section \ref{seccion.kakeya} is devoted to show some uniform bounds for the universal maximal operator and the Kakeya maximal operator over radial functions.

\bigskip

\section{The maximal operator over radial functions}

\bigskip

First, let us describe briefly the concepts and notation that we are going to deal with. We recall that $w\in A_1(\real^n)$ if and only if for some $C>0$ one has
\[
\frac{w(B_R(x))}{|B_R(x)|} \leq C \essinf_{y\in B_R(x)} w(y),\quad \mbox{a.e. } x\in\real^n,\ \forall R>0.
\]
Equivalently, $w\in A_1(\real^n)$ if
\[
Mw(x)\leq C\ w(x),\quad \mbox{a.e. } x\in \real^n.
\]
The smallest values of $C$ for which the previous inequalities hold will be denoted by $[w]_{A_1(\real^n)}$ and $[w]_{A_1(\real^n)}^\ast$ respectively. Both are usually called the $A_1$ constant of the weight.

\bigskip

For radial weights there is still another characterization of the $A_1$ class. We use that if $w$ is radial, then $Mw$ is pointwise comparable with $\mathcal A w$, where $\mathcal A$ is the maximal operator over centered rings given by
\[
\mathcal A u(x):= \sup_{0\leq a\leq |x|\leq b} \fint_{a\leq|y|\leq b} |u(y)|\,dy.
\]
More precisely one has:

\bigskip

\begin{lemma}\label{lemma.anillos}
Let $u(x)=u_0(|x|)$ be a radial function in $L_{\loc}^1(\real^n)$. Then there exists a constant $K_n$ only depending on the dimension such that for all $x\in\real^n$ one has 
\begin{equation}\label{comparacion.anillos}
\frac{1}{K_n}\ \mathcal A u(x) \leq Mu(x) \leq 2 \mathcal A u(x).
\end{equation}
\end{lemma}

\bigskip

When acting on radial functions $\mathcal A$ can be written as a one-dimensional maximal operator. Given a weight $v$ over $[0,\infty)$ the associated uncentered maximal operator is defined as
\[
\tilde M_v g(x)=\sup_{0\leq a\leq x\leq b}\ \frac{1}{v([a,b])}\int_a^b |g(y)|v(y)\,dy,
\]
for $g\in L_{\loc}^1([0,\infty),v)$. If $u(x)=u_0(|x|)$ is a radial function over $\real^n$, writing $v_n(t)=t^{n-1}$, note that $\mathcal A u(x) = \tilde M_{v_n} u_0(|x|)$.

\bigskip

As a consequence $w(x)=w_0(|x|)$ is in $A_1(\real^n)$ if and only if there exists a constant $C>0$ so that for a.e. $x$ one has
\[
\mathcal Aw(x)=\tilde M_{v_n}w_0(|x|)\leq C w_0(|x|)=Cw(x).
\] 
The smallest of such constants will be denoted by $(w)_{A_1(\real^n}$. With the usual arguments in weight theory one can see that the previous condition is equivalent to the existence of a constant $C>0$ so that for all intervals $\subset [0,\infty)$ one has
\[
\frac{w_0 v_n(I)}{v_n(I)}\leq C\essinf_{x\in I} w_0(x).
\]
The smallest of such constants is again $(w)_{A_1(\real^n)}$.

\bigskip

The inequality (\ref{comparacion.anillos}) has already appeared in the literature. In \cite{MenarguezSoria2} T. Me\-n\'ar\-guez and the second author used the inequality $Mu(x)\leq 2 \tilde M_{v_n} u_0(|x|)$ to prove (\ref{teo.orejas}) below. Inequality (\ref{comparacion.anillos}) in the form $c_n \mathcal Au(x)\leq Mu(x)\leq C_n \mathcal Au(x)$, with the constants depending on the dimension, was used by Duoandikoetxea, Moyua, Oruetxebarria and Seijo in \cite{DMOS}. From it they deduced that for any $p\geq 1$ the radial weights $w$ so that $M:L_{\rad}^p(w)\rightarrow L_{\rad}^p(w)$ boundedly 
are exactly the radial $A_p$ weights.

\bigskip

First, we point out that if $w\in A_1(\real^m)$ for some $m\in\mathbb N$, then $w\in A_1(\real^n)$ for all $n\geq m$ and $[w]_{A_1(\real^n)}^\ast$ and $(w)_{A_1(\real^n)}$ grow at most linearly with $n$.

\bigskip

\begin{lemma}\label{lemma.a1.siempre}
Let $w_0$ be a non-negative function over $[0,\infty)$ and set $w(x)=w_0(|x|)$ and $d\mu(x)=w(x)\,dx$. If for some $m\in\mathbb N$ one has $w\in A_1(\real^m)$ then $w\in A_1(\real^n)$ for all $n\geq m$ and $[w]_{A_1(\real^n)}^\ast\leq C_m \frac nm [w]_{A_1(\real^m)}^\ast$. The same holds for $(w)_{A_1(\real^n)}$.
\end{lemma}

\bigskip

These bounds are almost optimal, let us see this with an example. Consider $w(x)=(1-|x|)^{-\alpha}$ with $0<\alpha<1$. It is easy to see that $w\in A_1(\real^n)$ for all $n\in\mathbb N$. Note that in $\real^n$ one has
\begin{equation}\label{cuenta.gamma}
\frac{w(B_1(0))}{|B_1(0)|}= \frac1{\omega_{n-1}/n}\int_{B_1} (1-|x|)^{-\alpha} \,dx = n\ \int_0^1 (1-t)^{-\alpha}t^{n-1}\,dt  = n\ \frac{\Gamma(1-\alpha)\Gamma(n)}{\Gamma(n+1-\alpha)},
\end{equation}
where we have used the equality
\[
\frac{\Gamma(x)\Gamma(y)}{\Gamma(x+y)} = \int_0^1 (1-t)^{x-1} t^{y-1}\,dt.
\]
Since $\Gamma$ is logarithmically convex, writing $n+1-\alpha=\alpha n+(1-\alpha)(n+1)$ one has $\Gamma(n+1-\alpha)\leq \Gamma(n)^\alpha\Gamma(n+1)^{1-\alpha}$, which together with (\ref{cuenta.gamma}) says that
\begin{equation}\label{medida.bola.unidad}
\frac{w(B_1(0))}{|B_1(0)|} \geq  \Gamma(1-\alpha)\ n^{\alpha}
\end{equation}
But note that $\inf_{y\in B_1(0)} w(y)=1$. This gives $[w]_{A_1(\real^n)}^\ast \geq \Gamma(1-\alpha)\ n^\alpha$. Since $\alpha$ can be taken arbitrarily close to $1$, this shows that the upper bound for $[w]_{A_1(\real^n)}^\ast$ in Lemma \ref{lemma.a1.siempre} is near to be optimal. The same calculation works for $(w)_{A_1(\real^n)}$.

\bigskip

There are radial weights for which the $A_1$ constants remain uniformly bounded in dimension. Indeed, we have the following characterization.

\bigskip

\begin{prop}\label{prop.a1.decreciente}
Let $w(x)=w_0(|x|)$ be a radial weight. The following statements are equivalent:
\begin{enumerate}[a)]
\item\label{conda} There exists $N>0$ so that $[w]_{A_1(\real^n)}$, $[w]_{A_1(\real^n)}^\ast$ or $(w)_{A_1(\real^n)}$ are uniformly bounded for all $n>N$,
\item\label{condb} There exists a constant $C>0$ so that for all $n>N$, all $f\in L_{\rad}^1(w)$ and $\lambda>0$ we have
\[
w(\{x\in\real^n: \mathcal Af(x)>\lambda\})\leq \frac C\lambda \int_{\real^n}|f(x)|w(x)\,dx.
\]
\item\label{condc} $w_0$ is essentially constant over dyadic intervals and decreasing up to a constant. This means that there exist positive constants $\beta$ and $\eta$ so that
\begin{eqnarray*}
&&\esssup_{r\in[R,2R]} w_0(r)\leq \beta  \essinf_{r\in[R,2R]} w_0(r),\quad \forall R>0,\\
&&w_0(s)\leq \eta\ w_0(t),\quad\forall s\geq t\geq 0.
\end{eqnarray*}
\end{enumerate}
\end{prop}

\bigskip

Power weights with negative powers, that is $w_0(t)=t^{-\alpha}$ with $\alpha\geq 0$, are examples of weights with these properties. It is easy to check that they satisfy condition $\ref{condc})$ with $\beta=2^\alpha$ and $\eta=1$.

\bigskip

Moreover, for the weights satisfying the properties in Proposition \ref{prop.a1.decreciente}, the maximal operator $M$ acting over radial functions is weakly bounded in weighted $L^1$ with constants uniformly bounded in dimension. This was already known for $w\equiv1$. M.T. Men\'arguez and the second author proved in \cite{MenarguezSoria2} that for all radial $f$ over $\real^n$ and $\lambda>0$ one has
\begin{equation}\label{teo.orejas}
|\{x\in\real^n:Mf(x)>\lambda\}| \leq \frac 4\lambda\|f\|_{L^1(\real^n,dx)}.
\end{equation}
We point out the following extension, that is a corollary of part $\ref{condb})$ of Proposition \ref{prop.a1.decreciente} and (\ref{comparacion.anillos}).

\bigskip

\begin{teo}\label{teorema.radiales}
Let $w(x)=w_0(|x|)$ be a radial weight. Assume that there exist $N>0$ and $C>0$ so that $(w)_{A_1(\real^n)}<C$ for all $n>N$. Then for all radial $f$ over $\real^n$ with $n\geq N$ and $\lambda>0$ one has
\begin{equation}\label{acotacion.radiales}
w(\{x\in\real^n:Mf(x)>\lambda\}) \leq \frac {4C}\lambda\int|f(x)| w(x)\,dx.
\end{equation}
\end{teo}

\bigskip

Note that if $w$ is a radial weight so that (\ref{acotacion.radiales}) holds for radial functions over $\real^k$, then $w\in A_1(\real^k)$. If $w$ is decreasing up to a constant, the argument leading to (\ref{comparacion.bolas}) in the proof of Proposition \ref{prop.a1.decreciente} below, shows that $w$ is essentially constant over dyadic intervals. Hence, by Proposition \ref{prop.a1.decreciente} we are in the hypothesis of Theorem \ref{teorema.radiales} and we have (\ref{acotacion.radiales}) in $\real^n$ for all $n\geq k$ with a constant independent of $n$.

\bigskip

We finish remarking that for the weights $w$ characterized in Proposition \ref{prop.a1.decreciente} $Mw$ and $\mathcal Aw$ are comparable with constants independent of the dimension because one always has
\[
\mathcal Aw(x)\leq (w)_{A_1(\real^n)} w(x) \leq (w)_{A_1(\real^n)} Mw(x).
\]
It is easy to find examples of radially increasing functions, for instance $w_0=t^\alpha$ with $\alpha>0$, so that $\mathcal Aw(x)=Mw(x)=\infty$ for all $x$. Therefore $\mathcal Aw\leq C Mw$ with $C$ independent of the dimension does not imply any of the conditions $\ref{conda})$, $\ref{condb})$, $\ref{condc})$.

\bigskip

\section{Proofs of the main results}\label{seccion.pruebas}

\bigskip

We begin proving Proposition \ref{prop.a1.decreciente}. We mention that the equivalence $\ref{conda})\Leftrightarrow\ref{condc})$ already appeared in \cite{CriadoSoria} for $[w]_{A_1}$. Here we will use similar arguments. Among them, the method of differentiation through dimensions also presented in \cite{CriadoSoria} that is contained in the following Lemma.

\bigskip

\begin{lemma}\label{lemma.differentiation}
Take $w_0\in L_{\loc}^1([0,\infty),t^{N-1}\,dt)$ for some $N\geq 1$. Then, for almost every $T>0$ and for all $s\geq 0$ and $R>0$ so that $s^2+R^2=T^2$, if we take points $z^n \in \real^n$ with $|z^n|=s$ and we denote $B(z^n,R)=\{y\in\real^n: |z^n-y|< R\}$, the following holds
\[
\lim_{n\rightarrow\infty} \fint_{B(z^n,R)} w_0(|x|)\, dx = w_0(T).
\]
\end{lemma}

\bigskip

We now proceed to prove Proposition \ref{prop.a1.decreciente}.

\bigskip

\begin{proof}[Proof of Proposition \ref{prop.a1.decreciente}]
It is easy to see that 
\begin{eqnarray}
[w]_{A_1}^\ast &\leq& [w]_{A_1},\label{des1}\\ 
\,[w]_{A_1}^\ast &\leq& 2(w)_{A_1},\label{des2}
\end{eqnarray}

\bigskip

Let us prove $\ref{conda})\Rightarrow \ref{condc})$. Assume that for $n\geq N$ one has $[w]_{A_1(\real^n)}^\ast\leq C_\ast$ with $C_\ast$ independent of $n$. First, we prove that $w_0$ is decreasing up to a constant. Assume that $t\geq s\geq 0$ and for each $n\geq N$ take $x_n,y_n\in \real^n$ so that $x_n=\alpha y_n$ and $|x_n|=s$, $|y_n|=t$. Consider the ball $B_R(x_n)$ with $R^2=t^2-s^2$. By hypothesis we have, in the almost everywhere sense,
\[
\frac{w(B_R(x_n))}{|B_R(x_n)|}\leq C_\ast\ w(x_n)= C_\ast\ w_0(s).
\]
In view of Lemma \ref{lemma.differentiation} we take limits as $n\rightarrow\infty$ and obtain
\[
w_0(t)=w_0(\sqrt{s^2+R^2})\leq C_\ast\ w_0(s),
\]
for almost every $s\leq t$.

\bigskip

In order to ensure that $w_0$ is essentially constant over dyadic intervales, we only need to prove that for all  $R>0$, $w_0(R)\leq Cw_0(2R)$ with $C$ independent of $R$. Take $n\geq N$ and consider $x\in \real^n$ with $|x|=1$ and the balls $B=B_{R/2}((R/2)x)$, $B^\ast=B_{R/2}((3/2)x)$ and $B^{\ast\ast}=B_{R/2}((5/2)x)$. Since $w\in A_1(\real^n)$ it is doubling and for some constant $K$ we have $w(B)\leq K w(B^\ast) \leq K^2 w(B^{\ast\ast})$. By the decreasing property of $w_0$ that we have just proved, for all $y\in B$ we have $w_0(R)\leq C_\ast w(y)$ and for all $y\in B^{\ast\ast}$ we have $w(y)\leq C_\ast w_0(2R)$. This yields
\begin{equation}\label{comparacion.bolas}
w_0(R) \leq C_\ast\ \frac{w(B)}{|B|} \leq C_\ast K^2\ \frac{w(B^{\ast\ast})}{|B^{\ast\ast}|} \leq (C_\ast K)^2\ w_0(2R).
\end{equation}

\bigskip

We now proceed with $\ref{condc})\Rightarrow\ref{conda})$. In view of (\ref{des1}) and (\ref{des2}) it is enough to see that $(w)_{A_1}$ and $[w]_{A_1}$ are uniformly bounded. To see the first, consider an interval $[a,b]\subset [0,\infty)$. If $b\leq 2a$, we have
\[
\frac{w_0v_n([a,b])}{v_n([a,b])} \leq \esssup_{a\leq r\leq b} w_0(r)\leq \esssup_{a\leq r\leq 2a} w_0(r) \leq \beta \essinf_{a\leq r\leq 2a} w_0(r)\leq \beta \essinf_{a\leq r\leq b} w_0(r).
\]
If $b>2a$ we use the fact that under the hypothesis that $w_0$ is essentially constant over dyadic intervals for any $R>0$ we have
\begin{eqnarray*}\label{hardy}
w_0v_n\ ([0,R]) &=& \sum_{k=1}^\infty \int_{2^{-k}R}^{2^{-k+1}R} w_0(t)t^{n-1}\,dt\ \leq\ \sum_{k=1}^\infty \beta^k w_0(R)\ \frac{2^{(-k+1)n}-2^{-kn}}{n}\nonumber\\ &\leq& \frac{2^nR^n}{n}\  w_0(R)\ \sum_{k=1}^\infty \left(\frac{\beta}{2^n}\right)^k\ \leq\ 2\beta\ w_0(R)\ v_n([0,R]),
\end{eqnarray*}
the last inequality provided we take $n>N= \log_2 \beta+1$. Using this
\[
\frac{w_0v_n([a,b])}{v_n([a,b])} \leq \frac{w_0v_n([0,b])}{v_n([a,b])} \leq 2\beta w_0(b)\ \frac{v_n([0,b])}{v_n([a,b])},
\]
and this is all we need. Note that by the decreasing property of $w_0$ we have $w_0(b)\leq \essinf_{a\leq r\leq b} w_0(r)$ and
\[
\frac{v_n([0,b])}{v_n([a,b])} = \frac{b^n}{b^n-a^n} \leq \frac{b^n}{b^n-(b/2)^n}\leq 2.
\]

\bigskip

One proves the uniform boundedness of $[w]_{A_1}$ in a similar way (see \cite{CriadoSoria} for the details).

\bigskip

To see the equivalence $\ref{conda})\Leftrightarrow\ref{condb})$ observe that $\ref{condb})$ is equivalent to
\begin{equation}\label{nocentrado.uniforme}
w_0v_n\Bigl(\{r\geq0\,:\,\tilde M_{v_n} f_0(r)>\lambda\}\Bigr)\leq \frac C\lambda \int_0^\infty |f_0(r)|w_0(r)v_n(r)\,dr,
\end{equation}
with $C$ independent of $f_0$, $n$ and $\lambda$. For the implication $\ref{conda})\Rightarrow\ref{condb})$ assume that $(w)_{A_1}$ is uniformly bounded. Consider a radial function $f(x)=f_0(|x|)$ over $\real^n$ and $\lambda>0$. The set
$E_\lambda= \{r\geq0\,:\,\tilde M_{v_n} f_0(r)>\lambda\}$ is the union of all the intervals $I\subset [0,\infty)$ verifying
\[
\frac1{v_n(I)} \int_I |f_0(t)| v_n(t)\,dt> \lambda.
\]
We use Young's selection principle (see Lemma 4.2.1 in \cite{Garsia}) to obtain a subset $\mathcal I$ of such intervals so that $E_\lambda\subset \bigcup_{I\in \mathcal I}$ and $\sum_{I\in \mathcal I} \chi_I \leq 2$. 

\bigskip

For each $I\in A$ by our assumption we have
\begin{eqnarray*}
w_0v_n\,(I)&\leq& (w)_{A_1(\real^n)}\ v_n(I)\ \essinf_{I} w_0 \leq \frac{(w)_{A_1(\real^n)}}\lambda  \int_I |f_0(t)|\,v_n(t)\,dt\ \essinf_{I}w_0  \\&\leq&\frac{(w)_{A_1(\real^n)}}\lambda \int_I |f_0(t)|\,w_0(t)\,v_n(t)\,dt
\end{eqnarray*}
and hence,
\[
w_0v_n(E_\lambda) \leq \sum_{I\in \mathcal I} w_0v_n (I)\ \leq \frac{2(w)_{A_1(\real^n)}}\lambda \int_0^\infty |f_0(t)|\,w_0(t)\,v_n(t)\,dt
\]

\bigskip

We finish with the implication $\ref{condb})\Rightarrow\ref{conda})$. Consider intervals $J\subset I\subset [0,\infty)$, take $f_0=\chi_J$ and $\lambda=w_0v_n(J)/v_n(I)$. Then $I\subset E_\lambda$ and by (\ref{nocentrado.uniforme}) we have
\[
w_0v_n(I)\leq w_0v_n(E_\lambda) \leq \frac C\lambda º\ w_0v_n(J),
\]
or equivalently
\[
\frac{w_0v_n(I)}{v_n(I)} \leq C \frac{w_0v_n(J)}{v_n(J)}.
\]
Since $J$ is arbitrary, this implies that
\[
\frac{w_0v_n(I)}{v_n(I)}\leq C \essinf_{t\in I} w_0(t).
\]
\end{proof}

\bigskip

Our next goal is to show Lemma \ref{lemma.a1.siempre} and Theorem \ref{teo.duoandi.vega}. The technical parts of the proofs are summarized in the following lemmas. The first one, due to Stein (see \cite{SteinStromberg}), provides a method of rotations that reduces the dimension when estimating the mean values over balls.

\bigskip

\begin{lemma}\label{metodo.rotaciones} Let $k<n$ be natural numbers. For each $x=(x^1,\ldots,x^n)\in\real^n$ we call $x_1=(x^1,\ldots,x^k)\in \real^k$ and $x_2=(x^{k+1},\ldots,x^n)\in \real^{n-k}$. By abuse of the language we will write $x=(x_1,x_2)$. For any positive and measurable function $f$ on $\real^n$ one has
\[
\frac{\int_{|x|<R}f(x)\,dx}{\int_{|x|<R}\,dx} = \frac{\int_{SO(\real^n)}\int_{|x_1|<R}f(\tau(x_1,0))|x_1|^{n-k}\,dx_1\,d\tau} {\int_{|x_1|<R}|x_1|^{n-k}\,dx_1},
\]
where $SO(\real^n)=\{\tau\in\mathcal M_{n\times n}(\real):\tau\tau^{t}=\tau^{t}\tau=I\}$, i.e. the special orthonormal group in $\real^n$ equipped with its Haar measure.
\end{lemma}

\bigskip

Roughly speaking, it asserts that an integral mean over a ball in $\real^n$ can be transformed into an integral mean over a ball in $\real^k$ combined with all possible rotations in $\real^n$. In \cite{SteinStromberg} Lemma \ref{metodo.rotaciones} was used to obtain the following pointwise controls of the maximal function.

\bigskip

\begin{lemma}\label{lemma.acotaciones}
Let $k<n$ be natural numbers. With the notation of the previous Lemma \ref{metodo.rotaciones}, if $g$ is a function over $\real^n$ we write $g_{x_2}(x_1) := g(x_1,x_2) = g(x)$. Denoting by $M_m$ the maximal operator in $\real^m$, we have the following bounds:
\begin{eqnarray*}
a)\quad \quad M_n f(x) &\leq& \frac nk \int_{SO(n)} M_k\left[(f\circ \tau)_{\tau^{-1}(x)_2}\right](\tau^{-1}(x)_1)\,d\tau,\\
b)\quad \quad M_n f(x) &\leq& \int_{SO(n)} \mathcal M_k\left[(f\circ \tau)_{\tau^{-1}(x)_2}\right](\tau^{-1}(x)_1)\,d\tau,
\end{eqnarray*}
where $\mathcal M_k$ stands for the $k$-dimensional spherical maximal operator.
\end{lemma}

\bigskip

We will also employ a technical result for radial weights.

\bigskip

\begin{lemma}\label{lemma.dificil}
Let $w(x)=w_0(|x|)$ be a radial weight over $\real^k$ so that for some $p\geq1$ we have $w\in A_p(\real^k)$. For each $\rho>0$ consider the weights $w_{\rho}(x)=w_0(\sqrt{\rho^2+|x|^2})$. Then one has $w_{\rho} \in A_p(\real^k)$, and moreover there exists a constant $C_k>0$ only depending on  $k$ so that for all $\rho\geq0$ one has $[w_{\rho}]_{A_p(\real^k)}\leq C_k \, [w]_{A_p(\real^k)}$. As a consequence, there exists a constant $\tilde C_k>0$ only depending on $k$ and on $w$ so that for all $\rho\geq0$ and $f\in L^p(w_\rho)$ one has
\[
\|M_kf\|_{L^p(w_{\rho})} \leq \tilde C_k \|f\|_{L^p(w_{\rho})}.
\]
\end{lemma}

\bigskip

Now we are in conditions to prove Lemma \ref{lemma.a1.siempre}.

\bigskip

\begin{proof}[Proof of Lemma \ref{lemma.a1.siempre}]
Assume that $w(x_1)=w_0(|x_1|)$ is an $A_1(\real^k)$ weight. In view of Lemma \ref{lemma.acotaciones} part $a)$ and Lemma \ref{lemma.dificil} if $n\geq k$ and $x\in\real^n$ we have
\begin{eqnarray*}%
M_n w(x)&\leq& \frac nk \int_{SO(\real^n)}M_k [w_{|\tau^{-1}(x)_2|}](\tau^{-1}(x)_1)\,d\tau \\&\leq& C_k  [w]_{A_1(\real^k)}^\ast \frac nk\int_{SO(\real^n)} w_{|\tau^{-1}(x)_2|}(\tau^{-1}(x)_1)\,d\tau \\&=& C_k  [w]_{A_1(\real^k)}^\ast \frac nk\ w(x).
\end{eqnarray*}

\bigskip

The bound for $(w)_{A_1}$ is immediate once we observe that for any $[a,b]\subset [0,\infty)$ one has
\begin{eqnarray*}
\frac{w_0v_n([a,b])}{v_n([a,b])} &=& \frac{n}{b^n-a^n}\int_{a}^b w_0(t)t^{n-1}\,dt \\&\leq&  b^{n-k}\ \frac{b^k-a^k}{b^n-a^n}\ \frac nk\ \frac k{b^k-a^k}\int_a^b w_0(t)t^{k-1}\,dt\\&\leq& \frac nk\ \frac{w_0v_k([a,b])}{v_k([a,b])}.
\end{eqnarray*}
The second inequality uses that $b^{n-k}(b^k-a^k) \le b^n-a^n$, whenever $0\le a\le b$. From the previous calculation one deduces that $\tilde M_{v_n} w_0 \leq n/k\ \tilde M_{v_k} w_0$, which implies the bound for $(w)_{A_1}^\star$.
\end{proof}

\bigskip

In order to prove Theorem \ref{teo.duoandi.vega} we will use Lemma \ref{lemma.acotaciones}, part $b)$, where the maximal spherical operator appears. We recall that if $n\geq 2$, for a suitable smooth function $f$ the maximal spherical operator is defined as
\[
\mathcal M_nf(x)=\sup_{r>0}\frac1{\omega_{n-1}}\int_{\mathbb S^{n-1}} |f(x+ry)|\,d\sigma_{n-1}(y).
\]
This operator is known to be bounded on $L^p(\real^n)$ if and only if $p>n/(n-1)$. E.M. Stein proved this in \cite{Stein} in the case that $n\geq 3$, and J. Bourgain in \cite{Bourgain4} for $n=2$. This allows to define the maximal spherical operator over functions in $L^p(\real^n)$ with $p>n/(n-1)$.

\bigskip

We say that a weight $w$ is in the class $W_p(\real^n)$ if $\mathcal M_n$ is bounded on $L^p(w)$. If $w$ is radial we have the following relation with the $A_p$ classes.

\bigskip

\begin{lemma}\label{lemma.wp}
Let $\mu$ be a radial measure over $\real^n$ with density $w(x)=w_0(|x|)$. If for certain $k$ one has $w\in A_p(\real^{k})$ then there exists $m\geq k$ so that $w\in W_p(\real^m)$. Moreover $\|\mathcal M_m\|_{L^p(\real^m,w)\rightarrow L^p(\real^m,w)}$ is controlled by $\|M_k\|_{L^p(\real^k,w)\rightarrow L^p(\real^k,w)}$.
\end{lemma}

\bigskip

Both, Lemmas \ref{lemma.dificil} and \ref{lemma.wp} are proved below. Now we are ready to prove Theorem \ref{teo.duoandi.vega}.

\bigskip

\begin{proof}[Proof of Theorem \ref{teo.duoandi.vega}]
Since $w\in A_p(\real^k)$, by Lemma \ref{lemma.dificil}, for all $a\geq 0$ one has $w_a \in A_p(\real^k)$. Furthermore by Lemma \ref{lemma.wp} there exist $m\geq k$ so that for all $a\geq 0$ one has $w_{a}\in W_p(\real^m)$. Moreover $\|\mathcal M_m\|_{L^p(\real^m,w_a)\rightarrow L^p(\real^m,w_a)}$ is controlled by $\|M_k\|_{L^p(\real^k,w)\rightarrow L^p(\real^k,w)}$. This means that there exists $C_m>0$ so that for all $a\geq 0$ and $g\in L^p(w_a)$ one has
\[
\|\mathcal M_m g\|_{L^p(\real^m,w_{a})} \leq C_m \|f\|_{L^p(\real^m,w_{a})}.
\]
Given $f\in L^p(\real^n,w)$ with $n>m$, by Lemma \ref{lemma.acotaciones}, part $b)$ and Minkowski inequality we have
\begin{equation*}
\|M_nf\|_{L^p(\real^n,w)} \leq \int_{SO(n)} \left(\int_{\real^n} \left|\mathcal M_m[(f\circ \tau)_{\tau^{-1}(x)_2}](\tau^{-1}(x)_1)\right|^p w(x)\,dx\right)^{1/p}\,d\tau.
\end{equation*} 
Applying obvious changes of integration variables and Lemma \ref{lemma.dificil}, the previous is bounded by
\begin{eqnarray*}
\int_{SO(n)}\!\!\!\!\!&&\!\!\!\!\!\!\!\!\!\!\!\! \left(\int_{\real^n} \left|\mathcal M_m[(f\circ \tau)_{y_2}](y_1)\right|^p w(y)\,dy\right)^{1/p}\,d\tau \\&=&
 \int_{SO(n)} \left(\int_{\real^{n-k}}\int_{\real^k} \left|\mathcal M_m[(f\circ \tau)_{y_2}](y_1)\right|^p w_{|y_2|}(y_1)\,dy_1\,dy_2\right)^{1/p}\,d\tau \\ 
&\leq& \int_{SO(n)} \left(\int_{\real^{n-k}}C_m\int_{\real^k} \left|(f\circ \tau)_{y_2}(y_1)\right|^p w_{|y_2|}(y_1)\,dy_1\,dy_2\right)^{1/p}\,d\tau \\ &=& C_m \|f\|_{L^p( w)}.
\end{eqnarray*}
\end{proof}

\bigskip

We finish justifying Lemmas \ref{lemma.dificil} and \ref{lemma.wp}.

\bigskip

\begin{proof}[Proof of Lemma \ref{lemma.dificil}]
If $w\in A_p(\real^k)$ there exist $u$, $v\in A_1(\real^k)$ so that $w=uv^{1-p}$ and $[w]_{A_p(\real^k)}\leq [u]_{A_1(\real^k)}[v]_{A_1(\real^k)}^{p-1}$. Moreover if $w$ is radial, $u$ and $v$ can be chosen to be radial by their construction (see \cite{CuervaRubiodeFrancia}). Therefore it is enough to prove the result in the case $p=1$. 

\bigskip

Assuming $w\in A_1 (\real^k)$, we are going to show that there exist a constant $C>0$, so that for all $x\in\real^k$ and $\rho\geq 0$ one has $Mw_{\rho}(x)\leq C w_\rho(x)$. As we observed after Lemma \ref{lemma.anillos}, if $u(x)=u_0(x)$ is a radial and locally integrable function we have $Mu(x)\leq Cu(x)$ a.e. if and only if $\tilde M_{v_k} u_0(|x|)\leq C'u_0(|x|)$ a.e. By hypothesis this last condition is true for $w_0$. Now we check it for $w_0(\sqrt{\rho^2+(\,\cdot\,)^2})$. We take $0\leq a\leq |x|\leq b$. With the change of variables $s^2=\rho^2+t^2$ we obtain
\begin{eqnarray*}
\frac{k}{b^k-a^k}\int_a^b &&\!\!\!\!\!\!\!\!\!\!\!\!\!\!\!\!w_0(\sqrt{\rho^2+t^2})\,t^{k-1}\,dt\\&=& \frac{k}{b^k-a^k}\int_{\sqrt{\rho^2+a^2}}^{\sqrt{\rho^2+b^2}} w_0(s)\,(s^2-\rho^2)^{k/2-1}s\,ds\\&\leq& \frac{k^2/2}{(\rho^2+b^2)^{k/2}-(\rho^2+a^2)^{k/2}}\int_{\sqrt{\rho^2+a^2}}^{\sqrt{\rho^2+b^2}} w_0(s)\,s^{k-1}\,ds\\ &\leq& \frac k2\ \tilde M_{v_k} w_0(\sqrt{\rho^2+|x|^2}) \leq \frac k2\,(w)_{A_1(\real^k)}\ w_0(\sqrt{\rho^2+|x|^2}).
\end{eqnarray*}
The only step that is not immediate is the first inequality. Clearly it would follow from
\[
\left[(\rho^2+b^2)^{k/2}-(\rho^2+a^2)^{k/2}\right] \left(\frac{s^2-\rho^2}{s^2}\right)^{k/2-1} \leq \frac k2 (b^k-a^k).
\]
Observe that the left hand side of this inequality is increasing in $s$ for $s\geq \rho$. Thus, it is enough to check the case $s^2=\rho^2+b^2$, that is
\begin{equation}\label{casolim}
\left[(\rho^2+b^2)^{k/2}-(\rho^2+a^2)^{k/2}\right] \left(\frac{b^2}{\rho^2+b^2}\right)^{k/2-1} \leq \frac k2 (b^k-a^k).
\end{equation}
Applying the Mean Value Theorem to the function $\phi(t)=(\rho^2+t)^{n/2}$ yields
\[
(\rho^2+b^2)^{k/2}-(\rho^2+a^2)^{k/2} \leq \frac k2 (\rho^2+b^2)^{k/2-1}(b^2-a^2).
\]
This, together with the observation that $b^{k-2}(b^2-a^2)\leq b^k-a^k$, proves (\ref{casolim}).
\end{proof}

\bigskip

\begin{proof}[Proof of Lemma \ref{lemma.wp}]
We assume $w\in A_p(\real^k)$. By the Reverse H\"older property there exists $s<1$ so that $w^{1/s}$ is also an $A_p(\real^k)$ weight. Observe that $w^{1/s}\in A_p(\real^m)$ for all $m\geq k$. This is an easy consequence of the factorization and Lemma \ref{lemma.a1.siempre}.

\bigskip

Let us use the notation $\mathcal M_mf(x)=\sup_{t>0} S_t f(x)$, where
\[
S_tf(x)=\frac1{\omega_{m-1}}\int_{\mathbb S^{m-1}} |f(x+ty)|\,d\sigma_{m-1}(y).
\]

\bigskip

Following J.L. Rubio de Francia in \cite{RubiodeFrancia} we perform a dyadic decomposition of the spherical maximal function. Let $\psi:[0,\infty)\rightarrow[0,\infty)$ be a smooth function supported in $[1/2,2]$, so that for all $r>0$
\[
\sum_{j=-\infty}^\infty \psi(2^{-j} r)= 1.
\]
Let $\phi(r)=\sum_{j=-\infty}^{-1} \psi(2^{-j}r)$. We can define the operators $S_t^j$ and $B_t$ by $(S_t^j f)^\wedge(\xi)=(S_t f)^\wedge (\xi) \psi(2^{-j}|\xi|)$ and $(B_t f)^\wedge (\xi)=(S_tf)^\wedge \phi(2^{-j}|\xi|)$. It is obvious that
\begin{equation}\label{acotacion.suma}
\mathcal M_mf(x) = \sup_{t>0}\left|\sum_{j=0}^\infty B_tf(x)+S_t^j f(x)\right| \leq M_mf(x)+\sum_{j=0}^\infty \sup_{t>0} |S_t^jf(x)|.
\end{equation}
It is known that if $p<m/(m-1)$ then 
\[ 
\|\sup_{t>0}\left|S_t^jf\right|\|_{L^p(\real^m,dx)} \leq 2^{j(1-m/p')}\|f\|_{L^p(\real^m,dx)},
\]
and that $S_t^jf(x)\leq C2^{j} M_mf(x)$ (see \cite{RubiodeFrancia} and \cite{DuoandiVega}). Then since $w^{1/s}\in A_p(\real^m)$, one has
\[
\|\sup_{t>0}\left|S_t^jf\right|\|_{L^p(w^{1/s})} \leq C 2^{j}\|f\|_{L^p(w^{1/s})}.
\]
By interpolation with change of measure (see \cite{SteinWeiss}), one has
\[
\|\sup_{t>0}\left|S_t^jf\right|\|_{L^p(\real^m,w)} \leq C 2^{sj+(1-s)j(1-m/p')}\|f\|_{L^p(\real^m,w)}.
\]
If $m>p'/(1-s)$ the exponent in this bound is negative and then we can sum in $j$ to obtain
\begin{eqnarray*}
\|\mathcal M_mf\|_{L^p(\real^m,w)} &\leq& \|M_mf\|_{L^p(\real^m,w)} + \sum_{j=0}^\infty \Bigl\|\sup_{t>0}\bigl|S_t^jf\bigr|\Bigr\|_{L^p(\real^m,w)}\\ &\leq& C \|f\|_{L^p(\real^m,w)} + C \sum_{j=0}^\infty 2^{[s+(1-s)(1-m/p']j} \|f\|_{L^p(\real^m,w)}\\ &\leq& C \|f\|_{L^p(\real^m,w)}.
\end{eqnarray*}
Here $C$ may depend on $m$, $p$ and $w$ but is independent of $f$. $C$ is indeed  controlled by the operator norm of $M_m$ in $L^p(\real^m,w)$.
\end{proof}

\bigskip

\section{Kakeya maximal operator}\label{seccion.kakeya}

\bigskip

Fixed $N>0$, we denote by $\mathcal R_N$ the family of all parallelepipeds in $\real^n$ with edge lengths $h\times h\times \cdots\times h\times Nh$, where $h>0$ is arbitrary. The Kakeya maximal operator of eccentricity $N$ is defined as
\[
\mathcal K_Nf(x)=\sup_{x\in R\in \mathcal R_N}\frac1R\int_R |f(y)|\,dy.
\]
It is easy to prove that $\mathcal K_Nf(x)\leq N^{(n-1)}Mf(x)$ where $Mf$ is here the usual maximal function over all rotated cubes. One just has to replace $R\in\mathcal R_N$ by the smallest cube that contains it. Then $\mathcal K_N$ is weakly bounded on $L^1(\real^n)$ with a constant growing with $N$ at most at the rate $N^{n-1}$. By interpolation with the $L^\infty$ case the operator norm on $L^p(\real^n)$ grows at most like $N^{(n-1)/p}$ for $1<p<\infty$. However, it is conjectured that for $p=n$ it grows no faster than $C_\eps N^{\eps}$ for each $\eps>0$. A. C\'ordoba proved in \cite{Cordoba} that the conjecture is true in the case $n=2$. In higher dimensions A. Carbery, E. Hern\'andez and the second author showed in \cite{CarberyHernandezSoria} that the conjecture holds when restricting the action of $\mathcal K_N$ to radial functions. Alternative proofs and extensions are due to J. Duoandikoetxea, V. Naibo and O. Oruetxebarria  \cite{DuoandikoetxeaNaiboOruetxebarria} and J. Duoandikoetxea, A. Moyua and O. Oruetxebarria \cite{DuoMoyuaOruetxebarria}.

\bigskip

In this last three papers the result is obtained as a corollary of boundedness results for the universal maximal operator. This is defined as
\[
\mathcal Kf(x)=\sup_{u\in\mathbb S^{n-1}} \sup_{a\leq0 \leq b} \frac1{b-a}\int_{a}^b |f(x+su)|\,ds.
\]
$\mathcal K$ is related to the Kakeya maximal operator in the sense that it can be regarded as its extremal case, where the eccentricity $N$ is infinity and rectangles become segments. Moreover $\mathcal K$ majorizes all the $\mathcal K_N$ but turns out to be unbounded on every $L^p$, except for $p=\infty$ (see \cite{Guzman}). In spite of this, \cite{CarberyHernandezSoria} established that $\mathcal K:L_{\rad}^{n,1}(\real^n)\rightarrow L_{\rad}^{n,\infty}(\real^n)$.

\bigskip

The basic idea to give an alternative proof of this last bound in \cite{DuoandikoetxeaNaiboOruetxebarria} is that for $f=\chi_A$, the characteristic function of a radial set, we have $\mathcal Kf(x)\leq C_n\ (\mathcal Af(x))^{1/n}$. This was further refined in \cite{DuoMoyuaOruetxebarria} to obtain that for a $f$ radial $\mathcal Kf(x)\leq C_q (\tilde M_{v_2} f_0^q(|x|)^{1/q}$ for any $q>2$, and for $q\geq 2$ if $f$ is the characteristic function of a radial set. These last constants $C_q$ are independent of the dimension, although the weighted inequalities obtained from them were not. Here we prove

\bigskip

\begin{lemma}\label{lemma.kakeya.anillos}
Let $E$ be a radial subset of $\real^n$ and $f=\chi_E$. Then for all $k\geq 2$ one has the pointwise inequality
\[
\mathcal Kf(x) \leq 2( \tilde M_{v_k}f_0(|x|))^{1/k}.
\]
The constant 2 in this inequality is sharp.
\end{lemma}

\bigskip

As a consequence, via Proposition \ref{prop.a1.decreciente}, we obtain

\bigskip

\begin{teo}\label{teorema.kakeya.radial} Let $f$ be a radial function over $\real^n$, with $n\geq 2$, and let $w$ be a radial weight in $A_1(\real^n)$, then
\[
\|\mathcal Kf\|_{L^{n,\infty}(\real^n,w)}^\ast\leq \frac{2n}{n-1}\ [2(w)_{A_1(\real^n)}]^{1/n}\ \|f\|_{L^{n,1}(\real^n,w)}^\ast.
\]
\end{teo}

\bigskip

Observe that in view of Lemma \ref{lemma.a1.siempre} the previous implies a bound that is uniform in dimension. As a consequence one also has such a bound for the Kakeya maximal operator $\mathcal K_N$. We remark that the only original results claimed in this section are the sharp bound in Lemma \ref{lemma.kakeya.anillos} and the uniformity in the bound of Theorem \ref{teorema.kakeya.radial}.

\bigskip

Assuming Lemma \ref{lemma.kakeya.anillos} for the moment, we provide a proof of the above theorem.

\bigskip

\begin{proof}[Proof of Theorem \ref{teorema.kakeya.radial}]
By density, we just need to prove the result for a simple function of the form
\[
f(x)=\sum_{j=1}^J c_j\chi_{E_j}(x),
\]
where $E_1\supset\ldots\supset E_J$ are radial sets and $c_1,\ldots,c_J$ are positive reals.
If $E$ is a radial set, by Lemma \ref{lemma.kakeya.anillos} for $k=n$ and following the argument in the proof of Proposition \ref{prop.a1.decreciente} one has
\begin{eqnarray*}
w(\{x\in\real^n: \mathcal K\chi_E(x)>\lambda\}) &\leq& w(\{x\in\real^n: 2\mathcal A\chi_E(x)^{1/n}>\lambda\})\\ &\leq& \left(\frac2\lambda\right)^n\ 2(w)_{A_1(\real^n)}\ w(E).
\end{eqnarray*}
Hence $\|\mathcal K\chi_E\|_{L^{n,\infty}(\real^n,w)}^\ast \leq 2\,[2(w)_{A_1(\real^n)} w(E)]^{1/n}$. For a general $f$ we use the standard procedure:
\begin{eqnarray*}
\|Kf\|_{L^{n,\infty}(\real^n,w)}^\ast &\leq& \|Kf\|_{L^{n,\infty}(\real^n,w)} \leq \sum_{j=1}^J c_j\|K\chi_{E_j}\|_{L^{n,\infty}(\real^n,w)} \\&\leq& \frac{n}{n-1} \sum_{j=1}^Jc_j\|K\chi_{E_j}\|_{L^{n,\infty}(\real^n,w)}^\ast \\&\leq& \frac{2n}{n-1} [2(w)_{A_1(\real^n)}]^{1/n} \sum_{j=1}^Jc_jw(E_j)^{1/n}\\&=&\frac{2n}{n-1}[2(w)_{A_1(\real^n)}]^{1/n}\|f\|_{L^{n,1}(\real^n,w)}^\ast.
\end{eqnarray*}
\end{proof}

\bigskip

\begin{proof}[Proof of Lemma \ref{lemma.kakeya.anillos}]
We need some notation. Given $w,z\in\real^n$ we denote by $S_{w,z}$ the segment whose extremal points are $w,z$. We may assume that $|w|\leq|z|$ and will call $y$ to the point in $S_{w,z}$ which is closest to the origin. Consider a radial set $A\subset \real^n$, we define its radial projection over $S_{z,w}$ as $A_0=\{|y|\leq t\leq |z|: \exists x\in A \mbox{ with } |x|=t\}$. Denoting by $|E|_k$ the $k$-dimensional Hausdorff measure of a set $E$ it is enough to prove that one has
\[
\frac{|S_{w,z}\cap A|_1}{|z-w|} \leq 2 \left(\frac{k}{|z|^k-|y|^k}\int_{|y|}^{|z|}\chi_{A_0}(s)s^{k-1}\,ds\right)^{1/k}.
\]

\bigskip

We proceed in several steps, first we show how to reduce to the case $y=w$. Define $z'$ as the point aligned with $w$ and $z$ so that $|z'-y|=|z-y|$. Note that 
\[
\frac{|S_{w,z}\cap A|_1}{|z-w|} \leq \frac{|S_{z',z}\cap A|_1}{|z-y|} = 2\ \frac{|S_{y,z}\cap A|_1}{|z-y|}.
\]

\bigskip

For the second step we assume $y=w$. We call $L:=|z-y|$ and $\ell:=|S_{y,z}\cap A|_1$. Consider the point $u\in S_{y,z}$ so that $\ell:=|u-y|$. Defining $A^\ast=\{x\in \real^n: |y|\leq |x|\leq |u|\}$, 
we are done if we show that we have
\begin{eqnarray}\label{desigualdades}
\frac{|S_{y,z}\cap A|_1}{L} &=& \frac{|S_{y,z}\cap A^\star|_1}{L}\ \leq\ \left(\frac{k}{|z|^k-|y|^k}\int_{|y|}^{|z|}\chi_{A_0^\star}(s) s^{k-1}\,ds\right)^{1/k}\nonumber\\&\leq&  \left(\frac{k}{|z|^k-|y|^k}\int_{|y|}^{|z|}\chi_{A_0}(s)s^{k-1}\,ds\right)^{1/k}.
\end{eqnarray}
The equality in (\ref{desigualdades}) is a trivial consequences of the definition of $A^\star$. Now we prove the first inequality in (\ref{desigualdades}) in the case $k=2$. Denoting by $\gamma$ the angle determined by $S_{0,y}$ and $S_{y,z}$ at $y$, the inequality can be rewritten as
\[
\frac\ell L \leq \left(\frac{|u|^2-|y|^2}{|z|^2-|y|^2}\right)^{1/2}= \left(\frac{\ell^2-2\ell|y|\cos\gamma}{L^2-2L|y|\cos\gamma}\right)^{1/2},
\]
where the last expression comes from the cosine law. 
This is equivalent to
\[
\frac \ell L\leq \frac{\ell-2|y|\cos\gamma}{L-2|y|\cos\gamma},
\]
which is obviously true since $\ell\leq L$ and $\cos\gamma\leq0$. For $k\geq 2$ it is enough to show then that
\[
\left(\frac{|z|^k-|y|^k}{|u|^k-|y|^k}\right)^{1/k} \geq \left(\frac{|z|^2-|y|^2}{|u|^2-|y|^2}\right)^{1/2}.
\]
Dividing by $|y|$ and renaming $\alpha=(|z|/|y|)^2$ and $\beta=(|u|/|y|)^2$ the previous inequality becomes
\[
\left(\frac{\alpha^{k/2}-1}{\beta^{k/2}-1}\right)^{1/k} \geq \left(\frac{\alpha-1}{\beta-1}\right)^{1/2},
\]
or equivalently
\[
(\alpha-1)^{k/2}(\alpha^{k/2}-1)\geq(\beta-1)^{k/2}(\beta^{k/2}-1),
\]
which is true for $\alpha>\beta\geq1$ since $s\mapsto (s-1)^{k/2}(s^{k/2}-1)$ is clearly an increasing function for $s\geq1$.

\bigskip

As for the second inequality in (\ref{desigualdades}), let us define $T=\{|v-y|: v\in A\cap S_{y,z}\}$. Note that $|T|=\ell$ and that $T=\{s\geq 0: s-2|y|\cos\gamma \in A\}$. Therefore, by the change of variables $t=(s^2+|y|^2-2s|y|\cos\gamma)^{1/2}$ one has
\begin{eqnarray*}
\int_{A_0} t^{k-1}\,dt &=& \int_T (s^2+|y|^2-2s|y|\cos\gamma)^{(k-2)/2}(s-2|y|\cos \gamma)\,ds \\&\geq& \int_0^\ell (s^2+|y|^2-2s|y|\cos\gamma)^{(k-2)/2}(s-2|y|\cos \gamma)\,ds\\&=&\int_{|y|}^{|u|}t^{k-1}\,dt = \int_{A_0^\ast} t^{k-1}\,dt.
\end{eqnarray*}
To get the above inequality we have used that the integrated function is increasing with $s$.
\end{proof}

\bigskip

\begin{xrem}
The constant $2$ in this Lemma is optimal. To see this assume that $C>0$ is a constant such that for all $z\in\real^n$ and all radial set $A\subset\real^n$ one has
\begin{equation}\label{constante.anillos}
\mathcal K \chi_A(z) \leq C \tilde M_{v_2}\chi_{A_0}(|z|)^{1/2},
\end{equation}

\bigskip

Consider the segment $S_{w,z}$ and define $y$ as the point in $S_{w,z}$ that is closest to the origin. Assume that $|w|>|y|$ take $A=\{x\in\real^n: |y|\leq|x|\leq |w|\}$. By orthogonality
\[
\tilde M_{v_2} \chi_{A_0}(|z|) = \sup_{|y|\leq t\leq |w|} \frac{|w|^2-t^2}{|z|^2-t^2} = \frac{|w|^2-|y|^2}{|z|^2-|y|^2} = \frac{|w-y|^2}{|z-y|^2} = \frac{(\ell/2)^2}{(L-\ell/2)^2}
\]

\bigskip

Calling as before $L=|z-w|$ and $\ell=2|w-y|$ we have $\mathcal K\chi_A(z)\geq \ell/L$. Then inequality (\ref{constante.anillos}) implies
\[
C\geq 2\ \frac{L-\ell/2}L.
\]
Since, choosing $w$ appropriately, $\ell$ can be taken as small as wanted, necessarily we must have $C\geq2$.

\end{xrem}

\bigskip

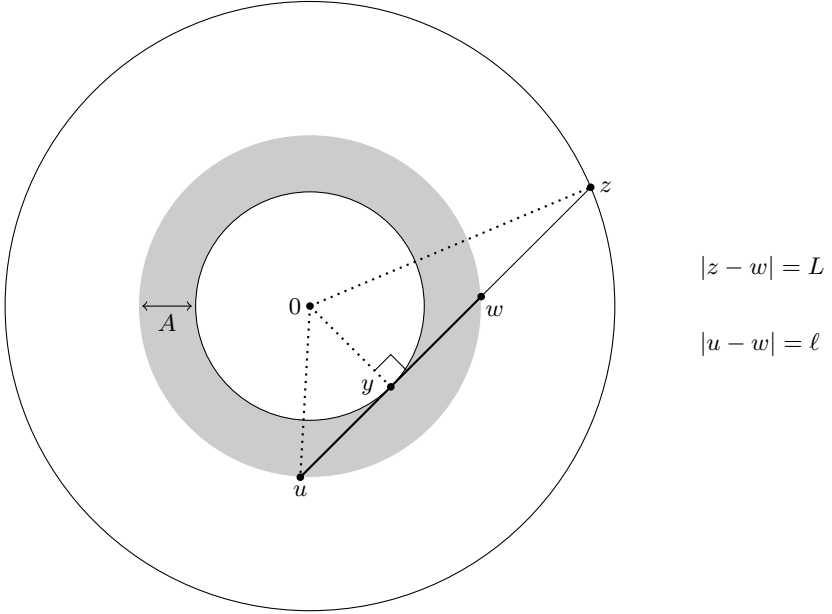
\begin{center}
\begin{tikzpicture}
\def\a{1.5}
\def\r{4}
\def\myangle{-45}
\pgfmathsetmacro{\radio}{0.7*\a+0.3*\r}%

\coordinate (O) at (0,0);
\coordinate (Y) at (\myangle:\a);

\coordinate (AUX1) at (\myangle-90:4);
\coordinate (AUX2) at (\myangle+90:4);
\coordinate (K1) at ($(Y)+(AUX1)$);
\coordinate (K2) at ($(Y)+(AUX2)$);
\path[name path=L] (K1)--(K2);
\path[name path=circulo] (O) circle (\r);
\path[name intersections={of=circulo and L}];
\coordinate (Z) at (intersection-1);
\draw[name path=circulo,color=white,fill=black!20] (O) circle (\radio);
\path[name intersections={of=circulo and L}];
\coordinate (W) at (intersection-1);
\coordinate (U) at (intersection-2);

\draw[fill=white] (O) circle (\a);

\draw (Z) node [right] {$z$} --(W) node [below,xshift=5pt] {$w$};

\draw [thick] (U)--(W);

\fill (O) node [left] {$0$} circle (0.5mm) (Y) node [left,xshift=-1mm] {$y$} circle (0.5mm) (Z) circle (0.5mm) (W) circle (0.5mm) (U) node [below] {$u$} circle (0.5mm);

\draw (O) circle (\r);

\draw (\r+1,0.5) node [right] {$|z-w|=L$};
\draw (\r+1,-0.5) node [right] {$|u-w|=\ell$};

\coordinate (BA) at (160:\a+0.05);
\coordinate (BB) at (160:\r-0.05);

%\draw [<->] (BA)--(BB) node [midway,above=0cm] {$A_{|y|,|z|}$};

\coordinate (CA) at (180:\a+0.05);
\coordinate (CB) at (180:\radio-0.05);

\draw [<->] (CA)--(CB) node [midway,below=0cm] {$A$};

\draw [dotted,thick] (O)--(Y) (O)--(U) (O)--(Z);

\def\distancia{0.3}

\coordinate (DA) at ($(Y)+(\myangle+90:\distancia)$);
\coordinate (DB) at ($(DA)-(\myangle:\distancia)$);
\coordinate (DC) at ($(Y)-(\myangle:\distancia)$);

\draw (DA)--(DB)--(DC);

\end{tikzpicture}
\end{center}

\bigskip

\paragraph{Acknowledgements.} Both authors were partially supported by DGU grant MTM2010-16518.

% ----------------------------------------------------------------


\begin{thebibliography}{99}

\bibitem{Aldaz} J.M. Aldaz, \emph{Dimension dependency of the weak type (1,1) bounds for maximal functions associated to finite radial measures.} Bull. Lond. Math. Soc. \textbf{39} (2007), 203--208.
\bibitem{Aldaz2} J.M. Aldaz, \emph{The weak type (1,1) bounds for the maximal function associated to cubes grow to infinity with the dimension.} Ann. of Math. (2) \textbf{173} (2011), no. 2, 1013--1023.
\bibitem{AldazPerezLazaro} J.M. Aldaz and J. P\'erez L\'azaro, \emph{Dimension dependency of $L^p$ bounds for maximal functions associated to radial measures.} Positivity \textbf{15} (2011), 199--213.
\bibitem{Aubrun} G. Aubrun, \emph{Maximal inequality for high-dimensional cubes.} Confluentes Math. \textbf{1} (2009), no. 2, 169--179.
\bibitem{Bourgain1} J. Bourgain, \emph{On high-dimensional maximal function associated to convex bodies.} Amer. J. Math. \textbf{108} (1986), no. 6, 1467--1476.
\bibitem{Bourgain2} J. Bourgain, \emph{On the $L\sp p$-bounds for maximal functions associated to convex bodies in $R\sp n$.} Israel J. Math. \textbf{54} (1986), no. 3, 257--265.
\bibitem{Bourgain3} J. Bourgain, \emph{On dimension free maximal inequalities for convex symmetric bodies in $\real^n$.} Geometrical aspects of functional analysis (1985/86), 168--176, Lecture Notes in Math., \textbf{1267}, Springer, Berlin, 1987.
\bibitem{Bourgain4} J. Bourgain, \textit{Averages in the plane over convex curves and maximal operators.} J. Analyse Math. \textbf{47} (1986), 69--85.
\bibitem{Bourgain5} J. Bourgain, \textit{On the Hardy-Littlewood maximal function for the cube} arXiv:1212.266v1.
\bibitem{Carbery} A. Carbery, \emph{An almost-orthogonality principle with applications to maximal functions associated to convex bodies.} Bull. Amer. Math. Soc. (N.S.) \textbf{14} (1986), no. 2, 269--273.
\bibitem{CarberyHernandezSoria} A. Carbery, E. Hern\'andez and F. Soria, \emph{Estimates for the Kakeya maximal operator on radial functions in $\real^n$}. Harmonic analysis (Sendai, 1990), 41–-50, ICM-90 Satell. Conf. Proc., Springer, Tokyo, 1991.
\bibitem{Cordoba} A. C\'ordoba, \emph{The Kakeya maximal function and the spherical summation multipliers.} Amer. J. Math. \textbf{99} (1977), no. 1, 1–-22.
\bibitem{Criado} A. Criado, \emph{On the lack of dimension free estimates in $L^p$ for maximal functions associated to radial measures.} Proc. Roy. Soc. Edinburgh Sect. A \textbf{140} (2010), no. 3, 541--552.
\bibitem{CriadoSjogren} A. Criado and P. Sj\"ogren, \emph{Bounds for maximal functions associated with rotational invariant measures in high dimensions} To appear in  J. Geom. Anal.
\bibitem{CriadoSoria} A. Criado and F. Soria, \emph{Localization and dimension free estimates for maximal functions}. J. Funct. Anal. \textbf{256}, no. 10, (2013), 2553--2583.
\bibitem{DuoMoyuaOruetxebarria} J. Duoandikoetxea, A. Moyua and O. Oruetxebarria, \emph{The spherical maximal operator on radial functions.} J. Math. Anal. Appl. \textbf{387} (2012), no. 2, 655–-666.
\bibitem{DuoandikoetxeaNaiboOruetxebarria} J. Duoandikoetxea, V. Naibo and O. Oruetxebarria, \emph{k-plane transforms and related operators on radial functions.} Michigan Math. J. \textbf{49} (2001), no. 2, 265--276.
\bibitem{DMOS} J. Duoandikoetxea, A. Moyua, O. Oruetxebarria and E. Seijo, \emph{Radial $A_p$ Weights with Applications to the Disc Multiplier and the Bochner-Riesz Operators.} Indiana Univ. Math. J. \textbf{57} (2008), no. 3, 1261--1281.
\bibitem{DuoandiVega} J. Duoandikoetxea and L. Vega, \emph{Spherical means and weighted inequalities.} J. London Math. Soc. (2) \textbf{53} (1996), no. 2, 343--353.
\bibitem{CuervaRubiodeFrancia} J. Garc\'{\i}a-Cuerva and J.L. Rubio de Francia, \emph{Weighted norm inequalities and related topics.} North-Holland Mathematics Studies, 116. Notas de Matem\'atica [Mathematical Notes], 104. North-Holland Publishing Co., Amsterdam, 1985.
\bibitem{Garsia} A.M. Garsia, \textit{Topics in almost everywhere convergence.}
Lectures in Advanced Mathematics, 4 Markham Publishing Co., Chicago, Ill. 1970.
\bibitem{Grafakos} L. Grafakos, \emph{Modern Fourier analysis.} Second edition. Graduate Texts in Mathematics, 250. Springer, New York, 2009.
\bibitem{Guzman} M. de Guzm\'an, \emph{Differentiation of Integrals in $\real^n$.} Lecture Notes in Math., Vol. 481. Springer-Verlag, Berlin-New York, 1975.
\bibitem{MenarguezSoria} M.T. Men\'arguez and F. Soria, \emph{Weak type (1,1) inequalities for maximal convolution operators.} Rend. Circ. Mat. Palermo (2) \textbf{41} (1992), no. 3, 342--352.
\bibitem{MenarguezSoria2} M.T. Men\'arguez and F. Soria, \emph{On the maximal operator associated to a convex body in $\real^n$.} Collect. Math. \textbf{3} (1992), 243--251.
\bibitem{Muller} D. M\"uller, \emph{A geometric bound for maximal functions associated to convex bodies.} Pacific J. Math. \textbf{142} (1990), no. 2, 297--312.
\bibitem{NaorTao} A. Naor and T. Tao \emph{Random martingales and localization of maximal inequalities}, J. Funct. Anal. \textbf{259} (2010), no. 3, 731--779.
\bibitem{RubiodeFrancia} J.L. Rubio de Francia, \emph{Maximal functions and Fourier transforms.} Duke Math. J. \textbf{53} (1986), no. 2, 395-–404.
\bibitem{Stein} E.M. Stein, \textit{Maximal functions. I. Spherical means} Proc. Nat. Acad. Sci. U.S.A. \textbf{73} (1976), no. 7, 2174--2175.
\bibitem{Stein1} E.M. Stein, \emph{The development of square functions in the work of A. Zygmund.} Bull. Amer. Math. Soc. (N.S.) \textbf{7} (1982), no. 2, 359--376.
\bibitem{SteinStromberg} E.M. Stein and J.O. Str\"omberg, \emph{Behavior of maximal functions in $R^n$ for large n.} Ark. Mat. \textbf{21} (1983), no. 2, 250--269.
\bibitem{SteinWeiss} E.M. Stein and G. Weiss, \emph{Introduction to Fourier analysis on Euclidean spaces.} Princ. Math. Ser., No. 32. Princeton University Press, Princeton, N.J., 1971.


\end{thebibliography}
\end{document}